\documentclass[11 pt, leqno]{article}

\usepackage[english]{babel}

\usepackage{lipsum}

\usepackage[letterpaper,top=2cm,bottom=2cm,left=3.5cm,right=3.5cm,marginparwidth=1.75cm]{geometry}

\usepackage{amssymb,amsfonts,amsmath,amsthm,amscd,epsfig, hyperref, framed, mathtools, enumerate}
\usepackage{graphicx}
\newtheorem{manualtheoreminner}{Theorem}
\newenvironment{manualtheorem}[1]{%
	\renewcommand\themanualtheoreminner{#1}%
	\manualtheoreminner
}{\endmanualtheoreminner}

\newtheorem{manualcorollaryinner}{Corollary}
\newenvironment{manualcorollary}[1]{%
	\renewcommand\themanualcorollaryinner{#1}%
	\manualcorollaryinner
}{\endmanualcorollaryinner}

\numberwithin{exercise}{section}

\numberwithin{notation}{section}
\newtheorem{lemma}{Lemma}
\numberwithin{lemma}{section}

\numberwithin{definition}{section}
\newtheorem{proposition}{Proposition}
\numberwithin{proposition}{section}

\numberwithin{theorem}{section}

\numberwithin{example}{section}

\numberwithin{question}{section}

\numberwithin{remark}{section}

\newtheorem{conjecture}{Conjecture}

\newcommand{\EE}{{\mathbb E}}

\newcommand{\CC}{{\mathbb C}}
\newcommand{\PP}{{\mathbb P}}

\newcommand{\HH}{{\mathcal H}}

\newcommand{\ZZ}{{\mathbb Z}}

\theoremstyle{definition}
\newcommand{\address}{{
		\bigskip
		\footnotesize
		
		 \textsc{Department of Mathematics, The Graduate Center of CUNY, 365 Fifth Ave., New York, NY 10016}\par\nopagebreak
		\textit{E-mail address:} \texttt{cchang1@gradcenter.cuny.edu}

}}

\title{Hybrid Statistics of a Random Model of Zeta over Intervals of Varying Length}

\author{Christine Chang}
\date{}

\begin{document}
	\maketitle
			\begin{abstract}
	Arguin, Dubach \&  Hartung recently conjectured that an intermediate regime exists between IID and log-correlated statistics for extreme values of a random model of the Riemann zeta function. For the same model, we prove a matching upper and lower tail for the distribution of its maximum. This tail interpolates between that of the two aforementioned regimes. We apply the result to yield a new sharp estimate on moments over short intervals, generalizing a result by Harper. In particular, we observe a hybrid regime for moments with a distinctive transition to the IID regime for intervals of length larger than  $\exp(\sqrt{\log \log T})$. 
\end{abstract}
	\section{Introduction}
	
	The Riemann zeta function is defined for $\Re (s) > 1$ where $s \in \CC$ as 
	
	\begin{equation} \label{eulerproduct}
				\zeta(s) = \sum_{n\geq 1}  \frac{1}{n^s} = \prod_{p \geq 1} \frac{1}{1-p^{-s}},
	\end{equation}
	and by analytic continuation otherwise. 
Fyodorov, Hiary \& Keating \cite{FHK}  and	Fyodorov \& Hiary \cite{FK}  were the first to study the local maximum of zeta over smaller, randomly chosen intervals.  Their conjecture gives an upper bound on the distribution of the local maximum of zeta over an interval of order one. 
Let $\tau$ be uniformly distributed on $[T, 2T]$, $y >0$, and $y=O \left(\frac{\log \log T}{\log \log \log T} \right)$. Arguin, Bourgade \& Radziwill (\cite{ABR}, \cite{ABR2}) prove the upper and lower bound of the conjecture and show that  
	\begin{equation}  \label{ABR2a}
		\PP \left( \max_{|h| \leq 1} |\zeta(1/2+i (\tau+h)) |  > \frac{ \log T}{ (\log \log T)^{3/4}} e^y \right) \asymp ye^{-2y}e^{-y^2/\log \log T}.
	\end{equation}

	There is however a discrepancy in the subleading order exponent and right tail when considering the maximum over intervals of size $(\log T)^{\theta}$, where $\theta > 0$ is fixed. Let $y=o \left(\frac{\log \log T}{\log \log \log T} \right)$. Then for $0 \leq \theta < 3$, Arguin \& Bailey \cite{AB} prove that 
	\begin{equation} 
		\PP \left(	\max_{|h| \leq (\log T)^{\theta}} |\zeta(1/2 + i(\tau + h))| >   \frac{ (\log T)^{\sqrt{1+\theta}}}{ (\log \log T)^{\frac{1}{4 \sqrt{1+\theta}}}} e^y \right) \ll e^{-2\sqrt{1+\theta}y} e^{-y^2/\log \log T}. \label{AB_theorem}
	\end{equation} 
The notation $\ll$ denotes big O notation, further explained in Section \ref{notation}. 
	Notice that as $\theta \to 0$, there is a jump discontinuity between \eqref{ABR2a} and \eqref{AB_theorem}  (i.e., when $\theta = 0$ and $\theta > 0$, respectively).  The interval of length $(\log T)^{\theta}$ approaches an interval of length one, but the subleading order exponent jumps from $\frac{1}{4\sqrt{1+\theta}}$ to $3/4$;  the right tail  gains a factor of $y$. The discrepancy between the subleading order terms was addressed recently for a random model of zeta -- Arguin, Dubach \& Hartung \cite{ADH} introduced the expression $\frac{1+2\alpha}{4\sqrt{1+\theta}}$ where $\alpha \in (0,1)$ and $\theta=t^{-\alpha}$. This term interpolates between $\frac{1}{4\sqrt{1+\theta}}$ and $3/4$. The same subleading order prefactor, $1+2\alpha$, was given in the random energy model setting \cite{Kis, SchmKis}. 
	
	 The main result of this paper (Theorem \ref{upperlowerbd}) addresses the right tail discontinuity. We continue the work of \cite{ADH} on a random model of $|\zeta|$, defined as  
	\begin{equation} \label{random model1} 
		R_T(h)= \exp\left({\sum_{p \leq T}}  \Re \dfrac{(G_p p^{-ih})}{p^{1/2}}\right),\end{equation}  where the $G_p$ are IID complex standard normal random variables, indexed by prime numbers $p$	(cf. Section 1.1). We give a matching upper and lower bound for the distribution of its maximum over intervals of varying length.  	Observe that the prefactor of the tail in Theorem \ref{upperlowerbd} interpolates between $1$ and $y$. 
	\begin{manualtheorem}{1}\label{upperlowerbd}
Let $T$ be large, let $y  >0$ and $y = O\left( \frac{\log \log T}{\log \log \log T} \right)$. Let $\theta= (\log \log T)^{-\alpha}$ with $\alpha \in (0,1)$. Then we have 
		\begin{equation*}
			\PP \left( \max_{|h| \leq (\log T)^{\theta} } R_T(h) > \frac{ (\log T)^{\sqrt{1+\theta}}}{ (\log \log T)^{ \frac{1+2\alpha}{4\sqrt{1+\theta}}}} e^y \right) \asymp \left( 1+ \frac{y}{ (\log \log T)^{1-\alpha}} \right) e^{-2\sqrt{1+\theta}y} e^{- \frac{y^2}{\log \log T}}.	\label{upperlowerbd2}
		\end{equation*} 
	\end{manualtheorem}	
Theorem \ref{upperlowerbd} proves the general case for when $y$ grows slowly with $T$. Observe that for $y$ at most $(\log \log T)^{1-\alpha}$, the tail is $\ll 	e^{-2\sqrt{1+\theta}y} e^{-y^2/\log \log T} $.  On the other hand, for $y$ in the large deviation regime, i.e. when $y$ is larger than $(\log \log T)^{1-\alpha}$, the tail is $\ll \left( 1+ \frac{y}{ (\log \log T)^{1-\alpha}} \right) e^{-2\sqrt{1+\theta}y} e^{- \frac{y^2}{\log \log T}}$.  A nontrivial consequence is a new normalization for the moments of zeta over short intervals. 	By Theorem \ref{upperlowerbd}, we prove the following corollary. 
	\begin{manualcorollary}{1}\label{moments corollary}
		For $\beta_c= 2\sqrt{1+\theta}$ and $\alpha \in (1/2, 1)$, we have for $A>0$, 
		\begin{equation}  \label{momentseqn}
		\frac{1}{(\log T)^{\theta}}	\int_{|h| \leq (\log T)^{\theta} } (R_T(h))^{ \beta_c}  \; dh \asymp_A
			\frac{(\log T)^{ \frac{\beta_c^2}{4}}}{(\log \log T)^{\alpha-1/2}},
		\end{equation}
		for all $t \in [T, 2T]$ except possibly on a subset of Lebesgue measure $\ll 1/A$. 	
		\label{momentstheorem}
	\end{manualcorollary} 
	In relation to other work, for $\theta=0$, a consequence of Theorem 1.9 from  \cite{SaksWebb} shows that $$\dfrac{\sqrt{\log \log T}}{ \log T} \int_{|h| \leq 1} | R_T(h)|^2 \; dh $$ converges in distribution as $T \to \infty$. Now, we note that Corollary \ref{moments corollary} also holds for $\alpha \in (0, 1/2]$, though the result is not sharp.  Arguin \& Bailey \cite{AB} prove an upper bound for all $\theta >0$ for the Riemann zeta function:  
	\begin{equation}  \label{IIDmoments}
		\frac{1}{(\log T)^{\theta}}		\int_{|h| \leq (\log T)^{\theta} } | \zeta( 1/2 + it + ih)|^{\beta_c}   \; dh \ll A  (\log T)^{\frac{\beta_c^2}{4}}. 
	\end{equation}
	
	This holds for all $t \in [T, 2T]$ except possibly on a subset of Lebesgue measure $\ll 1/A$. Hence, for $\alpha \in (0,1/2]$, we expect no correction for the random model.  
	
	Now, equation  \eqref{IIDmoments} corresponds to what is known for the IID regime where $\theta > 0$ and no correction is expected \cite{BKL}. But for the log-correlated regime ($\theta=0$),  Harper gives the correction $1/\sqrt{\log \log T}$:
	\begin{manualtheorem}{2} [Theorem 1 and Corollary 1 in \cite{Harper2}] Uniformly for all large $T$, we have 
		$$	\int_{|h| \leq 1} | \zeta( 1/2 + it + ih)|^2 dh \leq A \frac{ \log T}{\sqrt{\log \log T}}$$
		for all $t \in [T, 2T]$ except possibly on a subset of Lebesgue measure $\ll \frac{ (\log A) \wedge \sqrt{\log \log T}}{A}.$
	\end{manualtheorem}
	Corollary \ref{moments corollary} yields the normalization $1/(\log \log T)^{\alpha-1/2}$ for $\alpha \in (1/2, 1)$ and generalizes Harper's result in the setting of the random model \eqref{random model1} . This  correction interpolates between $1/\sqrt{\log \log T}$ and $1$, i.e. between the correction for the log-correlated regime and that of the IID regime.  In contrast, for $\alpha \leq 1/2$, i.e. for intervals that have length greater than $\exp(\sqrt{\log \log T})$, we observe a transition into the IID regime, as shown in Arguin \& Bailey's result \eqref{IIDmoments}. A similar threshold at $\alpha= 1/2$  was noted in the realm of random multiplicative functions \cite{Xu23}.
	
	In addition to proving a bound for the right tail of the maximum, we prove a bound for its left tail: 
	\begin{manualtheorem}{3} \label{leftbound}	\label{leftbound} Let $T$ be large, let $y < 0$ and $y = O\left( \frac{\log \log T}{\log \log \log T} \right)$. Let $\theta= (\log \log T)^{-\alpha}$ with $\alpha \in (0,1)$. Then we have 
		\small
		\begin{equation*} 
			\PP \left( \max_{|h| \leq (\log T)^{\theta} } R_T(h)  > \frac{ (\log T)^{\sqrt{1+\theta}}}{ (\log \log T)^{ \frac{1+2\alpha}{4\sqrt{1+\theta}}}} e^{y} \right)  = 1 - O\left(  \frac{ e^{2\sqrt{1+\theta}y}}{ 1-y } \right), \end{equation*}
		as $t \to \infty$. 
	\end{manualtheorem}
Theorem \ref{upperlowerbd} and Theorem \ref{leftbound} imply tightness of the sequence $$\max_{|h| \leq (\log T)^{\theta}} \log(R_T(h)) - \left( \sqrt{1+\theta}\log \log T - \frac{1+2\alpha}{4\sqrt{1+\theta}} \log \log \log T \right),$$ for large enough $T$.  The order of magnitude of the maximum was used to conjecture the size of  pseudomoments  \cite{Gerspach}. 
	
Finally, Theorem \ref{upperlowerbd} and Corollary \ref{moments corollary} motivate the following conjectures for the Riemann zeta function, respectively. The first part of Conjecture \ref{conjecture} extends the conjecture stated on Pg.~2 of \cite{ADH}. 
	\begin{conjecture} \label{conjecture}
		Let $0<\alpha<1$ and $\theta=(\log \log T)^{-\alpha}$. Then 
		$$ \max_{|h| \leq (\log T)^{\theta}} |\zeta(1/2 + i\tau +ih)| = \frac{ (\log T)^{\sqrt{1+\theta}}}{(\log \log T)^{ \frac{1+2\alpha}{4\sqrt{1+\theta}}}} e^{\mathcal{M}_T},$$
		where $\tau$ is uniformly distributed on $[T, 2T]$, and ($\mathcal{M}_T, T > 1)$ is a tight sequence of random variables converging as $T \to \infty$ to a random variable $\mathcal{M}$ with right tail $$\PP(\mathcal{M} > y) \sim  \left( 1+ \frac{y}{ (\log \log T)^{1-\alpha}} \right) e^{-2\sqrt{1+\theta}y} e^{- \frac{y^2}{\log \log T}}.$$
		
		Moreover, for $\beta_c= 2\sqrt{1+\theta}$ and $\alpha \in (1/2, 1)$, we have for $A>0$, 
		
		$$ \PP \left(\frac{1}{(\log T)^{\theta}} \int_{|h| \leq (\log T)^{\theta}} |\zeta(1/2 + it) |^{\beta_c} dt > A \frac{(\log T)^{ \frac{\beta_c^2}{4}}}{(\log \log T)^{\alpha - 1/2}} \right) \ll \frac{1}{A}. $$
		
	\end{conjecture}

	\subsection{The Random Model of $\log |\zeta|$ }
To study $\log |\zeta|$ on a short interval $I$ on the critical axis, we consider its random model: 
\begin{equation} \label{logR}
	\log(R_T(h))= {\sum_{p \leq T}}  \Re \dfrac{(G_p p^{-ih})}{p^{1/2}}, \;\;\; h \in I.
	\end{equation} 
To obtain \eqref{logR}, we first take the logarithm of the Euler Product \eqref{eulerproduct} at $s= 1/2 + i(\tau + h)$, where $\tau$ is chosen uniformly from $[T, 2T]$ and $h \in I$. Next, we perform a Taylor expansion and obtain 
	\begin{equation}\label{randomsum}
		\log \zeta(1/2+i(\tau + h))\approx	\log \prod_{p} \left( 1- \frac{1}{p^{1/2+i(\tau+h)}} \right)  \approx \sum_{p \leq T}  \dfrac{p^{-i\tau} p^{-ih}}{p^{1/2}}.
	\end{equation}
We then identify $p^{-i\tau}$ with $G_p$, as defined in \eqref{random model1}.  After taking the real part of the expression in \eqref{randomsum}, we obtain $\log ( R_T(h))$, the random model of $\log |\zeta|$. The process $\log(R_T(h))$ is Gaussian with mean 0 and variance that is approximately $(\log \log T)/2$. 

To evoke the branching structure of the process, we rewrite $\log (R_T(h))$ as the sum of IID increments by using a multiscale decomposition. To that end, write $J=e^{e^j}$, $K= e^{e^k}$ and $T= e^{e^{t}}$. For any $h \in I$,  consider the partial sum 
	\begin{equation} \label{Yk1}
S_j(h) = \sum_{m=1}^{j} Y_m(h), \;\;\;\text{for all} \;\; j\leq t \;\;\; \text{where} \;\;\;\; Y_m(h) = \sum_{e^{m-1} < \log p \leq e^m} \Re \frac{ (G_p p^{-ih})}{p^{1/2}}. 
\end{equation} 
The increments $Y_m$ are Gaussian and have an approximate variance of $1/2$ by Merten's estimate, which states that for any $x \geq 1$ and some constant $C > 0$, we have 
$$ \sum_{p \leq x} \frac{1}{p} = \log \log x + C + O\left(\frac{1}{\log x}\right).$$
Since the increments are not perfectly IID, the random model $S_j(h)$ is an approximate random walk for every $h$. Throughout the proofs, we primarily use the partial sum notation. Observe that $S_j(h)=\log (R_J(h))$. 
	
	We will often consider the process $S_j(h)$ and its covariance at varying subintervals of time as done in \cite{ADH}. Define the process $S_{k,l}(h)$  where $1 \leq k < l \leq t$:
	\begin{equation} \label{inbetweencovariance}
		S_{k, l} (h) =\sum_{m=k}^{l} Y_m(h), \;\;\; |h| \leq (\log T)^{\theta}.
	\end{equation}
Let $h, h^{'} \in [-(\log T)^{\theta}, (\log T)^{\theta}]$. To find the covariance $\EE[S_{k, l}(h), S_{k,l}(h^{'})]$, we perform straightforward computations and use the laws of $G_p$ to obtain 
	\begin{equation} \label{covariance}
		\EE[S_{k, l}(h) \cdot S_{k, l}(h^{'}) ] = \sum_{k \leq m < l}\; \sum_{e^{m-1} < \log p \leq e^m} \frac{\cos (h-h^{'}) \log p}{2p}.
	\end{equation}  
The following corollary shows that the covariance \eqref{covariance} depends on the distance between the two points $h, h^{'}$: 
	\begin{manualcorollary}{2}  Define $S_{k, l}(h)$ as in \eqref{inbetweencovariance}. For $h, h^{'} \in [-(\log T)^{\theta}, (\log T)^{\theta}]$, we have 
		\begin{equation}
			\EE[ S_{k, l}(h), S_{k, l}(h^{'})] = 
			\begin{cases}
				\frac{1}{2} (l-k) + O( e^{2l} |h-h^{'}|^2) & \text{if  } |h-h^{'} |  < e^{-l}   \\
				O ( e^{-k}|h-h^{'}|^{-1}) & \text{if } |h-h^{'}| > e^{-k}.
			\end{cases}
		\end{equation}
	\end{manualcorollary} 
	\begin{proof}
		Use \eqref{covariance} and apply Lemma 2.2 of \cite{ADH}. 
	\end{proof}
	Notice that when $|h-h^{'}| < e^{-l}$, the covariance is nearly equal to the variance; hence, two random variables $S_{k, l}(h)$ and $S_{k,l}(h^{'})$ are roughly the same. But when $|h-h^{'}| > e^{-k}$, the two random variables decorrelate as $h$ and $h^{'}$ grow further apart.

	\subsection{Proof Outline}
\indent To prove the upper bound of the right tail (Theorem \ref{upperlowerbd}), we adopt the method from Theorem 1.3 in \cite{ADH}, but perform a union bound over $\ll e^{k+t\theta}$ discrete points since $S_k(h) \approx S_k(h')$ whenever  $|h-h'| < e^{-k}$. We use a barrier that is similar in spirit to the one in \cite{ADH}, but include an extra term $(1-k/t)t^{1-\alpha}$ due to the larger interval size. This causes the barrier to be high and ineffective in constraining process values for $k \in [0, t-t^{\alpha})$. We thus omit the barrier for this range. The main contribution comes from the latter range of $k$ where applying the ballot theorem contributes significantly to the new tail bound. Details are provided in Section 2. 

The proof for the lower bound of the right tail (Theorem \ref{upperlowerbd}) and for the left tail bound  (Theorem \ref{leftbound}) uses a Paley-Zygmund inequality argument. Both proofs are similar to the argument in Section 4 of \cite{ADH}, but employ a different barrier and slope. We consider negative values of $y$  for Theorem \ref{leftbound}. 

For Corollary \ref{moments corollary}, a heuristic for the proof is provided on pg.~26 in \cite{AB}, but we give a rigorous argument in the spirit of Corollary 1.4 and use a  \textit{good }event that includes the barrier used in Theorem \ref{upperlowerbd} (cf. equation \eqref{goodevent EA}).  This modification induces an application of the tail bound from Theorem \ref{upperlowerbd}, which yields the new correction term. 

For general literature on extreme values of the Riemann zeta function, we refer the reader to sources such as \cite{Arguin}, \cite{ABH} and \cite{BK}. \\\\
	\textbf{Organization of results.} We prove the upper and lower bound of Theorem \ref{upperlowerbd} in Sections \ref{Section 2} and \ref{Section: LB of right tail}, respectively. We prove Theorem \ref{leftbound} in Section \ref{Section: LeftTail of Max}. Lastly, we prove Corollary \ref{moments corollary} in Section \ref{Section: Moments}. \\\\
	\textbf{Notation.} \label{notation} We state the notation commonly used throughout the paper. We write $f(x) = O(g(x))$ or $f(x) \ll g(x)$ if there exists a $C \in \mathbb{R}_{>0}$ and $x_0 \in \mathbb{R}_{> 0}$ such that $|f(x)| \leq C g(x)$ for all $x \geq x_0$. The expression $f(x) \asymp g(x)$ means that both $f(x) \ll g(x)$ and $f(x) \gg g(x)$ hold. We use $T$ to denote a large real number on the critical line. For ease of notation, we shorten the $\log \log$ expression and use $t = \log \log T$, $k = \log \log K$ and $l = \log \log L$; we use both lower and capital case letters interchangeably. Specific notation introduced in a proof will be explained at its given place.  \\\\
	\textbf{Acknowledgements.} 
	I would like to thank my advisor, Louis-Pierre Arguin, for his guidance and for the many helpful comments and suggestions. I would also like to thank the Mathematical Institute of the University of Oxford for hosting my academic visit during which this paper was completed. Partial support is provided by grants NSF CAREER 1653602 and NSF DMS 2153803.

	\section{Upper Bound of the Right Tail} \label{Section 2} 
	The upper bound of the right tail in Theorem \ref{upperlowerbd} is proved in Proposition \ref{ubrighttail}. The left bound of the right tail is proved in the next section. 
	
	\begin{proposition} \label{ubrighttail}
	Let $T$ be large, let $y  >0$ and $y = O\left( \frac{\log \log T}{\log \log \log T} \right)$. Let $\theta= (\log \log T)^{-\alpha}$ with $\alpha \in (0,1)$. Then we have   \begin{equation*} 
			\PP \left( \max_{|h| \leq (\log T)^{\theta} } R_T(h) > \frac{ (\log T)^{\sqrt{1+\theta}}}{ (\log \log T)^{ \frac{1+2\alpha}{4\sqrt{1+\theta}}}} e^y \right) \ll \left( 1+ \frac{y}{ (\log \log T)^{1-\alpha}} \right) e^{-2\sqrt{1+\theta}y} e^{- \frac{y^2}{\log \log T}}.
		\end{equation*} 
	\end{proposition}
	
	\begin{proof}	
		Let $\mu$ denote the slope 
		 \begin{equation}\label{slope} 
			\mu= \sqrt{1+\theta} - \dfrac{1+2\alpha}{4\sqrt{1+\theta}} \frac{\log t}{t}.
		\end{equation} 
Define 
		\begin{equation}\label{why 11/4} 
		b(k)=(1-k/t)t^{1-\alpha} +  \frac{11}{4 \alpha}\psi(k) +1, \;\; \text{for any }\;0\leq  k \leq t,
	\end{equation} 
where the logarithmic ''bump function" $\psi(k)$ is
\begin{equation}
	\psi(k)=  \begin{cases}
		\log (\min(k, t-k)), & \text{if  }1 \leq k \leq t-1 \\
		0 & \text{if  } k \in \{0, t\}.
	\end{cases} 
\end{equation} 
We define the barrier $M(k)$ to be 
		\begin{equation} \label{ub: barrier} 
			M(k) = \mu k + y + b(k), \;\; \text{for any }\;0\leq k \leq t.
		\end{equation} 
In the definition of $b(k)$, we multiply $\psi(k)$ by  $\frac{11}{4 \alpha}$  to ensure convergence of sums (cf. equations \eqref{firstsum3}, \eqref{why11/4 2}). 
		
		The key idea of the proof is to consider the \textit{first hitting time} of the process. We decompose the main event over when the maximum of the process first crosses the barrier, say at $k+1$: 
		\begin{align} 
			&\PP \left( \max_{|h| \leq (\log T)^{\theta}}S_t (h) > M(t)  \right) \notag\\
			&= \sum_{k=0}^{t-1} \PP ( \max_{|h| \leq (\log T)^{\theta}}S_t(h) > M(t),\;\; \max_{|h| \leq (\log T)^{\theta}} S_j(h) < M(j); \forall j \leq k , \notag \;\\ &\hspace{8cm}\max_{|h| \leq (\log T)^{\theta}} S_{k+1} (h) > M(k+1)). \label{hittingtimesum}
		\end{align}
		We drop the first event in \eqref{hittingtimesum} and estimate the following sum over two ranges of $k$: 
		\begin{align}
			\label{upperboundmainsum} 
			&\leq \sum_{k=0}^{t-1} \PP \left(  \max_{|h| \leq (\log T)^{\theta}} S_j(h) < M(j), \forall j \leq k, \;\max_{|h| \leq (\log T)^{\theta}} S_{k+1}(h) > M(k+1) \right). 
		\end{align} 
		We base the two cases on when the barrier $M(k)$ does not contribute to sharp estimates. Indeed, observe that at $k=0$, the barrier $M(k)$ is high at $t^{1-\alpha} +y$ and continues to far exceed the process values for $k \leq t-t^{\alpha}$. Hence, the barrier $M(j)$ for all $j \leq k$ is included only for $k > t-t^{\alpha}$. The dominant contribution comes from the latter range of $k$.   \\\\
	$\bullet \textbf{ Case 1}:  0 < k \leq t-t^{\alpha}$ \\
First, we omit the barrier for all $j \leq k$ from the event in \eqref{upperboundmainsum}. Next, we divide the interval $[ -(\log T)^{\theta}, (\log T)^{\theta}]$ into subintervals of length $e^{-(k+1) }$ and perform a union bound:
		\begin{align}
			&\sum_{k=0}^{t-t^{\alpha}} \PP \left( \max_{|h| \leq (\log T)^{\theta}} S_{k+1}(h) > M(k+1)  \right) \ll \sum_{k=0}^{t-t^{\alpha}} e^{k+1+t\theta}\;\PP \left( \max_{|h| \leq e^{-(k+1)}} S_{k+1}(h) > M(k+1)\right). \label{firstsum1}
		\end{align}
		To handle the probability on the right hand side of  \eqref{firstsum1}, we use the following lemma:
		
		\begin{lemma}[Lemma 3.6 of \cite{ADH}]
			\label{lemmasingleGaussian}
			Let $C>0$ and $1 \leq j \leq t$. For any $1 < y \leq C_j$, we have 
			$$\PP\left( \max_{|h|\leq e^{-j}} S_j(h) > y\right) \ll_C \frac{e^{-y^2/j}}{j^{1/2}},$$ 
			where $\ll_C$ means that the implicit constant depends on $C$. 
		\end{lemma}
		
		The essential observation of Lemma \ref{lemmasingleGaussian} is that the maximum of $S_{k+1}$ over a subinterval of size $e^{-(k+1)}$ behaves like a single Gaussian random variable with mean 0 and variance $k+1$. Hence, by Lemma \ref{lemmasingleGaussian}, the inequality \eqref{firstsum1} is bounded by
		\begin{align} \label{singlegaussian}	
			\ll \sum_{k=0}^{t-t^{\alpha}} e^{k+1+t\theta} \dfrac{1}{\sqrt{k+1}} \exp \left( - \dfrac{ (M(k+1))^2}{k+1} \right).
		\end{align}
The inequality \eqref{singlegaussian} simplifies to 
		\begin{align} 
			&\ll \sum_{k=0}^{t-t^{\alpha}} \dfrac{\exp \left( k+t\theta -\mu^2(k+1) - 2 \mu b(k+1) - 2\mu y - y^2/(k+1) \right)}{\sqrt{k+1}} && \notag \\
			&\ll \sum_{k=0}^{t-t^{\alpha}} \dfrac{ \exp ( (k-t)\theta  + \dfrac{1+2\alpha}{2\sqrt{1+\theta}} \log t \cdot \dfrac{k+1}{ t} - \frac{11}{2\alpha} \psi(k+1) - 2\sqrt{1+\theta}y - y^2/t  )}{ \sqrt{k+1}}.&& \label{firstsum2} 
		\end{align}
		We have the estimate 
		\begin{equation} \label{exp(-2uy) estimate} 
			\exp(-2\mu y) \ll e^{-2 \sqrt{1+\theta y}},
			\end{equation} since $y = O(\frac{t}{\log t}).$ Also, since $(\log x) / x$ is a decreasing function for $x > e$, we deduce that 
		\begin{equation}
			\label{eqn: decreasing log} 
			\dfrac{1+2\alpha}{2\sqrt{1+\theta}} \log t \cdot \dfrac{k+1}{ t} \leq  \dfrac{1+2\alpha}{2\sqrt{1+\theta}} \log (k+1).
		\end{equation}
		Then by \eqref{eqn: decreasing log}, and since $ e^{(k-t)\theta} \leq e^{-c}$  when $k < t-t^{\alpha}$ for some constant $c$, equation \eqref{firstsum2}  is
		\begin{align}
			&\ll e^{-2\sqrt{1+\theta} y} e^{-y^2/t}\; \sum_{k=0}^{t-t^{\alpha}} \frac{\exp \left(  \dfrac{1+2\alpha}{2\sqrt{1+\theta}} \log (k+1) - \frac{11}{2\alpha} \psi(k+1)  \right)}{\sqrt{k+1}}.  \label{firstsum3}
		\end{align} 
		Finally, since the log bump function $ \psi(k)$ is symmetric around $t/2$, we have: 
		\begin{equation}\label{logbump}
			\log( \min(k, t-k))=
			\begin{cases}
				\log k & \mbox{ if $0 < k < t/2$} \\
				\log (t-k)  &  \mbox{ if $ t/2 < k \leq t-t^{\alpha}$ }.
			\end{cases} 
		\end{equation} 
		After considering the two ranges in \eqref{logbump}, the sum \eqref{firstsum3} simplifies to $\ll e^{-2\sqrt{1+\theta}y} e^{-y^2/t}$. This concludes the upper bound estimate for $0 \leq k \leq t-t^{\alpha}$.
	\\\\
	$\bullet \textbf{ Case 2: }  t-t^{\alpha} < k \leq t$ \\\\
		We consider the sum \eqref{upperboundmainsum} for $t-t^{\alpha} < k \leq t$. By a union bound and translation invariance of the distribution, we have 	
		\begin{align} 
			&\sum_{k=t-t^{\alpha}}^{t} \PP \left(  \max_{|h|\leq (\log T)^{\theta}} S_j(h) < M(j); \forall j \leq k, \; \max_{|h|\leq (\log T)^{\theta}}  S_{k+1}(h) > M(k+1) \right) \notag \\ &\leq 	\sum_{k=t-t^{\alpha}}^{t} e^{k+1+t\theta}\; \PP \left( S_j(0) < M(j); \forall j \leq k ,\;\; \max_{|h|\leq e^{-{k}}} S_{k+1}(h) > M(k+1)  \right). \label{secondsum-mainsum}
		\end{align} 
	To find the probability in \eqref{secondsum-mainsum}, we first decompose the event over all possible values $u$ of the process at time $k$. 	For ease of notation, we define 
		\begin{equation}\label{set}
			\mathcal{B}_{k+1}(u) = \{ S_j(0) < M(j); \forall j \leq k,\; S_{k+1} (0) \in [u, u+1] \}.
		\end{equation} Using \eqref{set}, the probability in \eqref{secondsum-mainsum} is
		\begin{equation}
		\sum_{u \leq M(k)} \PP \left(	\mathcal{B}_{k+1}(u),  \;\; \max_{|h|\leq e^{-{k}}} S_{k+1}(h) > M(k+1)  \right). \label{twocases} 
		\end{equation}
	
	We will need two estimates for \eqref{twocases} since there are two cases for how the maximum of $S_{k+1}(h)$ over $|h| \leq e^{-k}$ can exceed  $M(k+1)$: either the difference $S_{k+1}(h) - S_{k+1}(0)$ is large for some $|h| \leq e^{-k}$ or  the increment $S_{k+1}(0) - S_k(0)$ is large.  Both cases yield the same estimate. We mainly focus on proving the second case. 
		
	For the first case, assume that the difference $S_{k+1}(h) - S_{k+1}(0)$ is large for some $|h| \leq e^{-k}$. Then, we have the estimate
\begin{align}
&	\sum_{k=t-t^{\alpha}}^{t} e^{k+1+t\theta}\;  \sum_{u \leq M(k)} \PP \left(	\mathcal{B}_{k+1}(u),  \;\; \max_{|h|\leq e^{-{k}}} S_{k+1}(h) - S_{k+1}(0) > M(k+1) - u  \right) \notag \\
		&\ll \left(1+\frac{y}{t^{1-\alpha}}\right)e^{-2\sqrt{1+\theta}y}e^{- y^2/t}. \label{firstcasesecondsum}
\end{align}
For the proof of the first case, we defer to Lemma 3.2 of \cite{ADH} where the same argument is used, but with modifications to the barrier and interval size. We use the barrier $M(k)$ as defined in \eqref{ub: barrier}, which changes the estimate in equation 3.34 of \cite{ADH}. 

For the second case, we assume that the increment $S_{k+1}(0) - S_k(0)$ is large. Then \eqref{twocases} is 
		\begin{align} 
			&\sum_{u \leq M(k)} \PP \left( \mathcal{B}_{k+1}(u), \;S_{k+1}(0)-S_k(0) > M(k+1) - u - 1  \right). \label{eqnindependence}
		\end{align} 
By independence of Gaussian increments, the sum \eqref{eqnindependence} is
		\begin{align} 
			&\sum_{u \leq M(k)} \PP (  \mathcal{B}_{k+1}(u) ) \cdot \PP\left( S_{k+1}(0)-S_k(0) > M(k+1) - u - 1  \right). \label{secondsumtwoprob} 
		\end{align} 
 It remains to find the two probabilities in \eqref{secondsumtwoprob}. For the first probability in \eqref{secondsumtwoprob}, an application of the Ballot Theorem \ref{thm: ballot linear UB} gives 
		\begin{align} \PP \left( S_j(h) < M(j); \forall j \leq k, \;S_k \in [u, u+1]  \right) \ll \dfrac{e^{-u^2/k}}{k^{3/2}} (t^{1-\alpha} + y) (M(k) -u-1). \label{secondsumfirstprob}
		\end{align} 
		For the second probability in \eqref{secondsumtwoprob}, a Gaussian estimate yields
		\begin{align} 
			\PP\left( S_{k+1}(0)-S_k(0) > M(k+1) - u - 1  \right) \ll e^{- (M(k+1)-u-1)^2}. \label{secondsumsecondprob}
		\end{align} 
		After substituting the two probability bounds \eqref{secondsumfirstprob} and \eqref{secondsumsecondprob} into the sum \eqref{secondsumtwoprob}, we have 
		\begin{align}
			&\ll  \sum_{u \leq M(k)}  (M(k) -u-1) \cdot (t^{1-\alpha}+y)k^{-3/2}  e^{- (M(k+1)-u-1)^2-u^2/k} \notag \\
			&\leq \sum_{v \geq 0} (t^{1-\alpha}+y) (v-1) \cdot k^{-3/2} e^{-\frac{(M(k)-v)^2}{k} - (v+O(1))^2}. \label{secondsumprobestimate}
		\end{align} 
		A change of variable is performed in line \eqref{secondsumprobestimate} with $v=M(k) - u$ denoting the distance from the process $S_k(h)$ to the barrier $M(k)$. (Recall that $u$ represents all possible values of $S_k(h)$.) This concludes the estimate for the probability in \eqref{secondsum-mainsum}.

		All that remains is to estimate the sum \eqref{secondsum-mainsum}. Substituting the probability bound \eqref{secondsumprobestimate} into  \eqref{secondsum-mainsum}, we obtain 
		\begin{align}
			&\leq \sum_{k=t-t^{\alpha}}^{t} e^{k+1+t\theta} \left( \sum_{v \geq 0}  \dfrac{e^{-\mu^2 k - 2b(k) - 2\mu y- \frac{y^2}{k}+2\mu v}}{k^{3/2}} (v-1)  (t^{1-\alpha}+y)  e^{-( v+O(1))^2} \right) \notag \\
			&\ll (t^{1-\alpha}+y)e^{-2\sqrt{1+\theta}y} e^{- y^2/t } \sum_{k=t-t^{\alpha}}^{t} \frac{e^{k+1+t\theta}  e^{-\mu^2 k-2b(k)}}{ k^{3/2}} \sum_{v \geq 0} (v-1) e^{2v- (v+O(1))^2}. \label{needssimplifying}
		\end{align}
		Observe that the sum over $v$ is finite. After simplifying \eqref{needssimplifying}, we obtain
		\begin{align}
			&\ll(t^{1-\alpha}+y)e^{-2\sqrt{1+\theta}y} e^{- y^2/t } \sum_{k=t-t^{\alpha}}^{t} \frac{ e^{-(t-k)\theta + \frac{ 1+2\alpha}{2} \log k - \frac{11}{2\alpha}\psi(k) } }{k^{3/2} }\notag \\
			&\ll t^{1-\alpha} \left(1+\frac{y}{t^{1-\alpha}}\right)e^{-2\sqrt{1+\theta}y} e^{- y^2/t } \sum_{k=t-t^{\alpha}}^{t}  k^{\alpha - 1} e^{-(t-k)\theta  -  \frac{11}{2\alpha} \log (t-k)} \notag \\
			&\ll  \left(1+\frac{y}{t^{1-\alpha}}\right)e^{-2\sqrt{1+\theta}y} e^{- y^2/t }  \sum_{k=t-t^{\alpha}}^{t}  \left(t/k\right)^{1-\alpha} (t-k)^{- \frac{11}{2\alpha}}  \label{why11/4 2}  \\
			&\ll \left(1+\frac{y}{t^{1-\alpha}}\right)e^{-2\sqrt{1+\theta}y} e^{- y^2/t }. \label{secondcaseestimate}
		\end{align}
	The estimates for the first \eqref{firstcasesecondsum} and second case  \eqref{secondcaseestimate} are the same.  This concludes the upper bound estimate for $t-t^{\alpha} < k \leq t$. 
	\end{proof}

	\section{Lower Bound of the Right Tail} \label{Section: LB of right tail}
	In this section, we prove a lower bound to Theorem \ref{upperlowerbd}. 
	\begin{proposition} \label{lbprop} 
		Let $T$ be large, let $y > 0 $ and $y = O\left( \frac{\log \log T}{\log \log \log T} \right)$. Let $\theta= (\log \log T)^{-\alpha}$ with $\alpha \in (0,1)$. Then we have 
		\begin{equation*} 
			\PP \left( \max_{|h| \leq (\log T)^{\theta} } e^{X_T(h)} > \frac{ (\log T)^{\sqrt{1+\theta}}}{ (\log \log T)^{ \frac{1+2\alpha}{4\sqrt{1+\theta}}}} e^y \right) \gg \left( 1+ \frac{y}{ (\log \log T)^{1-\alpha}} \right) e^{-2\sqrt{1+\theta}y} e^{- \frac{y^2}{\log \log T}}.
		\end{equation*} 
		
	\end{proposition}
Before we prove Proposition \ref{lbprop}, we introduce a few definitions. Consider the following set of discrete points  
	\begin{equation} \label{H}
		\mathcal{H} = [-e^{t\theta}, e^{t\theta}] \cap e^{-t}\mathbb{Z},
	\end{equation} 
	separated by a distance of $e^{-t}$. 
	To prove Proposition \ref{lbprop}, we use the Paley-Zygmund inequality to compare the first and second moment of the \textit{number of exceedances }
	\begin{equation} 
		Z_{\delta} = \sum_{h \in \mathcal{H}} 1_{J(h)} \label{exceedances}
	\end{equation} 
	of the event 
	\begin{align*} J(h) = \{     S_t(h) \in [ \mu t + y, \mu t + y + \delta], \;S_k(h) < M(k); \forall \; k \in [1,t] \cap \mathbb{Z} \}\end{align*}  
	for some fixed $\delta > 0$. Recall that the slope $\mu= \sqrt{1+\theta} - \frac{1+2\alpha}{4\sqrt{1+\theta}} \frac{\log t}{t}$. Define the barrier  $M(k)$ to be
	\begin{align} M(k) = \mu k  + \left(1-\frac{k}{t} \right)t^{1-\alpha} +y +1,\;\; \text{for any} \;\; 0 \leq k \leq t.
	 \end{align}

	We establish bounds for the first and second moment of $Z_{\delta} $ in the following two lemmas: 
	\begin{lemma} \label{lemma: lowerboundlemma1} Let $Z_{\delta}$ be as defined in $(\ref{exceedances})$ with $\delta > 0$. Then we have 
		\[ \EE[Z_\delta] \gg \left(1 + \frac{y}{t^{1-\alpha}}\right) e^{-2\sqrt{1+\theta} y} e^{-y^2/t},\]
		where $y(t) \to \infty.$
	\end{lemma}
	
	\begin{lemma} \label{lemma: lowerboundlemma2} Let $Z_{\delta}$ be as defined in $\eqref{exceedances}$ with $\delta > 0$. Then we have 
		$$\EE[Z_{\delta}^2] \ll  \left(1+\frac{y}{t^{1-\alpha}}\right) e^{-2\sqrt{1+\theta} y} e^{- y^2/t}, $$
		where $y(t) \to \infty$. 
	\end{lemma}
	The proof of Lemma \ref{lemma: lowerboundlemma1} is a straightforward application of the Ballot Theorem \ref{thm: ballot linear LB}. We mainly focus on the proof of Lemma \ref{lemma: lowerboundlemma2}.
	\begin{proof}[Proof of Proposition \ref{lbprop}]
		By Lemma \ref{lemma: lowerboundlemma1}, Lemma \ref{lemma: lowerboundlemma2} and the Paley-Zygmund inequality, we have 
		$$\PP(Z_{\delta} \geq 1) \geq \dfrac{ (\EE[Z_{\delta}])^2}{\EE[Z_{\delta}^2]} \gg \left(1+ \frac{y}{t^{1-\alpha}} \right) e^{-2\sqrt{1+\theta}y}e^{-y^2/t}.$$
	\end{proof}
	
	\begin{proof} [Proof of Lemma \ref{lemma: lowerboundlemma1}]
		By linearity of expectation, Fubini's Theorem, and translation invariance of the distribution, we have 
		$$	\EE[Z_\delta] 	= 2e^{t+t \theta} \;\PP(J(0)).  $$
		By the Ballot Theorem \ref{thm: ballot linear LB}, we obtain
		\begin{align} 
			\label{eq:1} 
			\PP (J(0)) \gg  \frac{ t^{1-\alpha} + y}{t} \left( \frac{ e^{\frac{- (\mu t + y)^2}{t}}}{\sqrt{t}}  \right) = \left(1+ \frac{y}{t^{1-\alpha}}\right) t^{- 1/2-\alpha}	e^{ -\frac{(\mu t + y)^2}{t}}. 
		\end{align}
		The exponential in \eqref{eq:1}  simplifies to 
		 \begin{align*} e^{ -\frac{(\mu t + y)^2}{t}} =  e^{- (1+\theta)t} t^{1/2+\alpha} e^{-O( (\log^2t )/ t)} e^{-2\mu y - y^2/t}.
		\end{align*} 
		It follows that 
		\begin{align*} 
			\PP(J(0)) \gg \left(1+ \frac{y}{t^{1-\alpha}}\right) e^{- (1+\theta)t}  e^{-2\mu y - y^2/t}. 
		\end{align*}
	Recall that $\exp(-2\mu y) \ll e^{-2 \sqrt{1+\theta y}}$ since $y = O(\frac{t}{\log t}).$ It follows that 
		\begin{align*}
			\EE[Z_\delta] = 2e^{t+t \theta} \;\PP(J(0)) \gg \left(1 + \frac{y}{t^{1-\alpha}}\right) e^{-2\sqrt{1+\theta} y} e^{\frac{-y^2}{t}}.
		\end{align*}

	\end{proof}

	\begin{proof} [Proof of Lemma \ref{lemma: lowerboundlemma2}]
		
		We first express the second moment as a sum over all possible pairs $(h, h^{'})$ in $I=[-e^{t\theta}, e^{t\theta}]$:  
		\begin{align}
			\label{mainexpression}
			\EE[(Z_{\delta})^2] &= \sum_{\substack{(h,h^{'}) \in I } } \PP(J(h) \cap J(h^{'}) ) \leq (\#\HH)  \sum_{|h| \leq e^{t\theta}} \PP(J(h) \cap J(0)).
		\end{align}
		By translation invariance of the distribution, we let $h^{'}=0$. For precise estimates, we divide $|h-h^{'}| < e^{t \theta}$ into three cases to evoke the covariance structure of the processes. The three cases are $1 < |h-h^{'}| \leq e^{t\theta/2}$,  $|h-h^{'}| > e^{t\theta/2},$ and $|h-h^{'}| \leq 1$. The main contribution comes from $|h-h^{'}| \leq 1$.  In particular, there is a coupling or decoupling of walks in each case for which we use the following lemma.
		
		\begin{lemma}[Lemma 7 of \cite{ABR2} ]
			\label{lemma: lowerboundlemma3}
			Let $|\rho| < \mathfrak{s}^2$. Consider the following Gaussian vectors and their covariance matrices: 
			\begin{align*}
				&(\mathcal{N}_1, \mathcal{N}_1^{'}),   &\mathcal{C}_1=  \left (
				{\begin{array}{cc}
						\mathfrak{s}^2 & \rho \\
						\rho & 	\mathfrak{s}^2 \\
				\end{array} } \right), \\
				&(\mathcal{N}_2, \mathcal{N}_2^{'}), &\mathcal{C}_2=  \left (
				{\begin{array}{cc}
						\mathfrak{s}^2 + |\rho| & 0 \\
						0 & 		\mathfrak{s}^2 + |\rho|  \\
				\end{array} } \right).
			\end{align*}
			Then for any measurable set $A \subset \mathbb{R}^2$, we have 
			$$\PP( (\mathcal{N}_1, \mathcal{N}_1^{'}) \in A) \leq \sqrt{ \frac{ \mathfrak{s}^2 + |\rho|}{ \mathfrak{s}^2 - |\rho|}} \cdot \PP( (\mathcal{N}_2, \mathcal{N}_2^{'}) \in A). $$
		\end{lemma}
		The proof of Lemma \ref{lemma: lowerboundlemma3} is straighforward and entails expanding the covariance matrix. See Lemma 7 of \cite{ABR2} for more details. \\\\
		$\bullet \textbf{ Case 1: } 1 < |h-h^{'}| \leq e^{t\theta / 2}$\\\\
		Observe that there are $\ll e^{t\theta / 2 + t}$ number of $h'$s of length $e^{-t}$ in an interval of length $2(e^{t\theta/2} - 1).$  Hence,  
		\begin{align} \label{case1 for lb} 
			(\# \mathcal{H}) \sum_{1 < |h-h^{'}| \leq e^{t\theta / 2}} \PP(J(h) \cap J(h^{'}))  \ll e^{t(1+\theta)}e^{t\theta / 2 + t} \; \PP(J(h) \cap J(0)).
		\end{align}
		We use translation invariance of the distribution in \eqref{case1 for lb}. To find $\PP(J(h) \cap J(0))$, we use Lemma \ref{lemma: lowerboundlemma3} and the Ballot Theorem \ref{thm: ballot linear UB}: 
		\begin{align}
			\PP(J(h) \cap J(0)) &\ll \PP(J(0))^2 \notag \\
			&\ll \left( \frac{ (t^{1-\alpha} + y + 1)}{t} \frac{ \exp (- (\mu t + y)^2/t)}{t^{1/2}} \right)^2  \notag \\
			&\ll \left( \frac{ (t^{1-\alpha} + y + 1)}{ t^{3/2}} \exp( -\mu^2 t + \frac{1+2\alpha}{2} \log t - 2 \mu y - y^2/t) \right)^2  \notag \\
			&\ll \left(\left(1 + \frac{y}{t^{1-\alpha}}\right) \exp( -t (1+\theta) - 2\sqrt{1+\theta}y - y^2/t) \right)^2. \label{eqn: probest1} 
		\end{align}
		After applying the probability estimate \eqref{eqn: probest1} to the inequality \eqref{case1 for lb}, we obtain 
		\begin{align} \label{case1lb}
			(\# \mathcal{H}) \sum_{1 < |h | \leq e^{t\theta / 2}} \PP(J(h) \cap J(0)) \ll \left(\left( 1 + \frac{y}{t^{1-\alpha}} \right) e^{-2 \sqrt{1+\theta} y} e^{- y^2/t} \right)^2  e^{- t\theta /2} . 
		\end{align}
		Recall that $ \EE[Z_\delta] \gg \left(1 + \frac{y}{t^{1-\alpha}}\right) e^{-2\sqrt{1+\theta} y} e^{-y^2/t}$. Hence, the estimate \eqref{case1lb} is  $o(\EE[Z_{\delta} ]^2).$ \\\\
		$\bullet \textbf{ Case 2: }  |h-h^{'}| > e^{t\theta / 2}$ \\\\
		We proceed as in the previous argument. Since there are $2e^{t+t \theta}(1-o(1)) \ll e^{t+t\theta}$ number of $h'$s of length $e^{-t}$ in an interval of length $2(e^{t\theta} - e^{t\theta/2})$, we obtain
		\begin{align}
			(\# \mathcal{H}) \sum_{ |h | > e^{t\theta / 2}} \PP(J(h) \cap J(0))  \ll\left(\left( 1 + \frac{y}{t^{1-\alpha}} \right) e^{-2 \sqrt{1+\theta} y} e^{- y^2/t} \right)^2 .
		\end{align}
		This is $O(\EE[Z_{\delta} ]^2).$
		\\\\
		$\bullet \textbf{ Case 3: } |h-h^{'}| \leq 1$\\\\
	 We handle case 3 differently by studying subintervals that depend on the branching time $k_b= \log |h-h^{'}|^{-1}$. More precisely, we consider subintervals where $e^{-k_b-1}< |h - h^{'}| \leq  e^{-k_b}$ for $0 < k_b \leq t$.  Define $(\EE[Z_{\delta}^2] )_{k_b}$ to be the contribution from such intervals:
		\begin{align} \label{branchtimesecondmoment}
		(\EE[Z_{\delta}^2] )_{k_b} = (\# \mathcal{H}) \sum_{e^{-k_b-1}<|h| \leq e^{-k_b} } \PP(J(0) \cap J(h))
		\end{align}
	so that 
	\begin{equation*}
			(\# \mathcal{H}) 	\sum_{|h| \leq 1} \PP(J(0) \cap J(h)) = 	\sum_{k_b} 	(\EE[Z_{\delta}^2] )_{k_b}. 
		\end{equation*}
		(By translation invariance of the distribution, we let $h^{'}=0$.) Observe that the number of $h$'s such that $|h| \approx e^{-k_b}$ in a discretized interval of order one is $e^t \cdot e^{-k_b}$. Continuing from \eqref{branchtimesecondmoment} we have
		\begin{align}
			\label{case3main}
			&\ll e^{t+t\theta} e^{t-k_b}  \;\; \PP(J(h) \cap J(0)).
		\end{align}
	
	We will estimate $\PP(J(h) \cap J(0))$ in two cases, which cover two ranges of $k_b$.  The first case covers the range $1 \leq k_b < t- t^{\alpha +\epsilon}$ for some fixed $\epsilon$ with $0 < \epsilon < 1-\alpha$. The second case covers remaining $k$. We omit the barrier in the first case for the same reason given in Section \ref{Section 2}.   The dominant contribution comes from the latter range of $k$.

			We will now prove the two cases. \\\\
		$\bullet \;\; \text{The contribution from} \;\; |h-h^{'}| \leq 1 \;\;\text{for} \;\;1 \leq k_b < t-t^{\alpha+\epsilon}.$ \\\\
		To find the probability in \eqref{case3main} for $1 \leq k_b < t-t^{\alpha+\epsilon}$, we drop the barrier from $J(h)$ and $J(0)$: 
		\begin{align}			
			\label{eqn: case3a}
			\PP(J(h) \cap J(0))	&\ll \PP(S_t(h) > \mu t + y, S_t(0) > \mu t + y).
		\end{align} 
		It is convenient to rewrite the two processes $S_t(h)$ and $S_t(0)$ in terms of their center of mass $\overline{S}_k=\frac{S_k(h) + S_k(0)}{2}$ and difference $S^{\perp}_k= \frac{ S_k(h) - S_k(0)}{2}$. This is because $S_k(0) \approx S_k(h)$ for $k \leq k_b$.
		Let $S_{k_b}(h) = q_1$ and $S_{k_b}(0)=q_2$ so that the average $v=\frac{q_1+q_2}{2}$ and difference $d=\frac{q_1- q_2}{2}$. We rewrite $S_t(h)$ and $S_t(0)$: 
		\begin{align}
			\label{eqn: centerofmass1}
		S_t(h) = q_1 + S_{k_b, t}(h) = v+d + S_{k_b, t}(h)
		\end{align}
		and 
		\begin{align}
			\label{eqn: centerofmass2}
			S_t(0) = q_2 + S_{k_b, t}(0) = v-d + S_{k_b,t}(0).
		\end{align}
		Using \eqref{eqn: centerofmass1} and \eqref{eqn: centerofmass2}, the probability in  \eqref{eqn: case3a} 
		\begin{align}
			\label{eqn: centerofmass3}
			&= \PP(v+d + S_{k_b, t}(h) \in (\mu t+y, \mu t + y + \delta],  v-d + S_{k_b, t}(0) \in (\mu t+y, \mu t + y + \delta]). \hfill
		\end{align} 
		Next, we decompose the probability \eqref{eqn: centerofmass3} over all possible averages $v$ and differences $d \in \ZZ$ at the branching time $k_b$:
		\begin{align} 
			&= \sum_{v < M(k_b)} \sum_{d} \PP( \overline{S}_{k_b} \in (v,v+1], S_{k_b}^{\perp} \in (d, d+1]) \cdot \label{eqn:k}  \\
			&\phantom{{}=1}	\PP(S_{k_b,t}(h) + v+d \in (\mu t+y, \mu t + y + \delta], S_{k_b,t}(0) + v-d \in (\mu t+y, \mu t + y + \delta]).  \notag 
		\end{align}
		The two probabilities in \eqref{eqn:k} are by the independence of random walk increments. To find the first probability in \eqref{eqn:k}, we use Lemma \ref{lemma: lowerboundlemma3} to obtain 
		\begin{align} 
			\label{case3aprob1} 
			\PP( \overline{S}_{k_b} \in (v,v+1], S_{k_b}^{\perp} \in (d, d+1])\ll \frac{ e^{-v^2/k_b}}{\sqrt{\pi k_b} }e^{-cd^2},
		\end{align} 
		where $c < 0$. Similarly, for the second probability in \eqref{eqn:k}, we have 
		\begin{align}
			&\PP(S_{k_b, t}(h) + v+d \in (\mu t+y, \mu t + y + \delta], S_{k_b, t}(0) + v-d \in (\mu t+y, \mu t + y + \delta])  \notag\\
			&\ll \frac{ \exp{ - \frac{(-v-d+\mu t + y)^2}{t-k_b}}\exp{ - \frac{(-v+d+\mu t + y)^2}{t-k_b}}}{\pi (t-k_b)}. 	\label{case3aprob2}
		\end{align}
		We are now in a position to find the contribution to 
		$(\EE[ Z_{\delta}^2])_{k_b}$ for $1 \leq k_b < t-t^{\alpha+\epsilon}$. By the probability estimates in \eqref{case3aprob1} and \eqref{case3aprob2}, we have 
		\begin{align*} 
			(\EE[ Z_{\delta}^2])_{k_b} &\ll e^{t+t\theta} e^{t-k_b}\;\; \PP(J(h) \cap J(0)) \\
			&\ll e^{t+t\theta} e^{t-k_b}\;\;\sum_{v < M(k_b)} \;\sum_{d \in \ZZ} \frac{ e^{-v^2/k_b}}{\sqrt{k_b} }e^{-cd^2} \frac{ e^{ - \frac{(-v-d+\mu t + y)^2}{t-k_b}}e^{ - \frac{(-v+d+\mu t + y)^2}{t-k_b}}}{(t-k_b)}.
		\end{align*} 
		By a change of variable with $\overline{v} = v- \mu( k_b)$, we obtain
		\begin{align}
			&\ll e^{t+t\theta} e^{t-k_b}\sum_{ \overline{v} < (t-k_b)\theta+y+1}  \sum_{d}  \frac{e^{- \frac{ (\overline{v} + \mu k)^2}{k_b}}}{\sqrt{k_b}} e^{-cd^2} \frac{ e^{- \frac{(\mu(t-k_b) + y-\overline{v}+d)^2}{t-k_b}}     e^{- \frac{(\mu(t-k_b) + y-\overline{v}-d)^2}{t-k_b}}    }{t-k_b} \notag \\
			&\ll \frac{e^{t+t\theta} e^{t-k_b} e^{-2\mu^2(t-k_b) -\mu^2(k_b)  - 4\mu y}}{\sqrt{k_b}(t-k_b)} \sum_{ \overline{v} < (t-k_b)\theta+y+1} e^{2\mu \overline{v} - \frac{2(y-\overline{v})^2}{t-k_b} - \frac{\overline{v}^2}{k_b}} \sum_d e^{-cd^2-\frac{2d^2}{t-k_b}}. 	\label{eqn: case3aCOV}
		\end{align} 
		The sum over $d$ is finite; to handle the sum over $\overline{v}$, we use Laplace's method. Let $f(\overline{v}) = 2\mu \overline{v} - \frac{2(y-\overline{v})^2}{t-k_b} - \frac{\overline{v}^2}{k_b}$. Now the maximum of $f(\overline{v})$ is where $k_b$ is nearly close to $t$. Subbing in $t-1$ for $k_b$, we find that $f(\overline{v})$ has a stationary point at $\overline{v_0} = \frac{ \mu + 2y}{2 + \frac{1}{t-1}}.$
		Hence, by Laplace's method, we obtain 
		\begin{align}
			\label{eqn: laplace}
			\sum_{ \overline{v} < (t-k_b)\theta+y+1} e^{f(\overline{v})} \leq \int_{-\infty}^{(t-k_b)\theta + y + 1} e^{f(\overline{v})} \; d \overline{v}  
			\ll e^{f(\overline{v_0})} \ll e^{2 \sqrt{1+\theta} y } e^{- y^2/t}.
		\end{align}
		After applying the estimate \eqref{eqn: laplace} to \eqref{eqn: case3aCOV}, the contribution is now
		\begin{align}
			(\EE[ Z_{\delta}^2])_{k_b}	&\ll \frac{e^{t+t\theta} e^{t-k_b} e^{-2\mu^2(t-k_b) -\mu^2(k_b)  - 4\mu y}}{\sqrt{k_b}(t-k_b)}  e^{2\sqrt{1+\theta} y} e^{ - y^2/t}. \notag
		\end{align}
	This simplifies to
		\begin{align} 
			&\ll \frac{e^{- \theta(t-k_b) + (1+2 \alpha )\log (t-k_b) + \frac{1+2\alpha}{2} \log k_b} e^{-2\sqrt{1+\theta}y} e^{- y^2/t}}{ \sqrt{k_b} (t-k_b)} \notag \\
			&\ll e^{- \theta(t-k_b) + 2 \alpha \log (t-k_b) + \alpha \log k_b} e^{-2\sqrt{1+\theta} y} e^{- y^2/t}.	\label{eqn: case3aeqn4} 
		\end{align} 
		In line \eqref{eqn: case3aeqn4}, we use the fact that $(\log x)/x$ is a decreasing function for $x > e$. 
		
		It remains to sum over $k_b$ up to $t-t^{\alpha + \epsilon}$ for a fixed $\epsilon$ with $0 < \epsilon < 1-\alpha$. (We choose  $t-t^{\alpha+\epsilon}$ to ensure sum convergence.) After performing an index change by letting $\hat{k} = t-k_b$, we obtain 
		\begin{align} \label{case3A estimate} 
			\sum_{k_b=1}^{t-t^{\alpha + \epsilon}}  	(\EE[ Z_{\delta}^2])_{k_b} &= \sum_{\hat{k} = t^{\alpha + \epsilon}}^{t-1} e^{- \theta(\hat{k}) + 2 \alpha \log \hat{k} + \alpha \log (t-\hat{k})} e^{-2\sqrt{1+\theta} y - y^2/t}  \ll e^{-2\sqrt{1+\theta} y} e^{ - y^2/t}.
		\end{align}
		This proves the case for $1 \leq k_b < t-t^{\alpha + \epsilon}$. The estimate \eqref{case3A estimate} is $o( \EE[Z_\delta]^2)$. \\\\
	$\bullet  \;\; \text{The contribution from} \;\;  |h-h^{'}| \leq 1 \;\;\text{for} \;\; t-t^{\alpha+\epsilon} \leq k_b \leq t$. \\\\		
		We proceed similarly to the case for $1 \leq k_b < t-t^{\alpha +\epsilon}$. We continue from \eqref{eqn:k}, but preserve the barrier constraint in the events $J(h)$ and $J(0)$: 
		\begin{align}
			 &(\# \mathcal{H}) \sum_{e^{-k_b-1}<|h| < e^{-k_b} } \PP(J(h) \cap J(0)) \notag \\
			&=e^{t+t\theta} e^{t-k_b}\;\; \sum_{v < M(k_b)} \sum_{d} \PP( \overline{S}_{k_b} \in (v,v+1], S_{k_b}^{\perp} \in (d, d+1], S_{k_b} \leq M(k_b)) \cdot\notag \\
			&\PP(S_{k_b, \;t}(h) + v+d \in (\mu t+y, \mu t + y + \delta], S_{k_b, \;t}(0) + v-d \in (\mu t+y, \mu t + y + \delta],  \notag  \\
			&\phantom{{}=1} S_{k_b, \;l}(h) < M(l),\;S_{k_b, \;l}(0) < M(l);\;\; \forall l \in [k_b+1, t] \cap \mathbb{Z}). \label{eqn: case3beqn1}
		\end{align}
		
We will estimate the two probabilities in \eqref{eqn: case3beqn1} over two ranges of $v$.  Recall that $v$ denotes the average between $S_{k_b}(0)$ and $S_{k_b} (h)$. Observe that the barrier at the branching time is much greater than any negative value of $v$. Hence, we omit the barrier for $v < 0$, but preserve it for $v \geq 0$.  
		
		Observe that for the first probability in  \eqref{eqn: case3beqn1}, the independence of $\overline{S}_{k_b}$ and $S^{\perp}_{k_b}$ gives
		\begin{align} \label{firstprobabilitycase3B} 
			&\PP( \overline{S}_{k_b} \in (v,v+1],  \overline{S}_{k_b} \leq M(k_b))\; \PP(S_{k_b}^{\perp} \in (d, d+1]).
		\end{align} We will now consider both $v > 0$ and $v \leq 0$ for \eqref{firstprobabilitycase3B}. For $v < 0$, we use a Gaussian estimate to obtain 
		\begin{align*}
			&\PP( \overline{S}_{k_b} \in (v,v+1])\; \PP(S_{k_b}^{\perp} \in (d, d+1]) \ll e^{-v^2}e^{-cd^2},
		\end{align*} 
		for some constant $c > 0$. 
		For $v \geq 0$, we preserve the barrier and use the Ballot Theorem \ref{thm: ballot linear UB} to obtain
		\begin{align}
			\label{eqn: case3b prob1 barrier}
			&\ll \frac{ ( t^{1-\alpha} + y+ 1) (\mu k_b + b(k_b) - v + 1)}{k_b} \frac{\exp (- v^2/ k_b)}{\sqrt{k_b}}\;e^{-cd^2}.
		\end{align} 
		The concludes the estimate for the first probability in \eqref{eqn: case3beqn1}. 
		
		It remains to find the second probability in \eqref{eqn: case3beqn1}. For ease of notation, let $m_y(t) = \mu t + y$ and let $[[ k_b+1, t]]=[k_b +1, t] \cap \mathbb{Z}$. By an application of Lemma 1.3, we obtain 
		\begin{align}
			&\PP(S_{k_b, \;t}(h) + v+d \in (m_y(t), m_y(t) + \delta], S_{k_b, \; t}(0) + v-d \in ( m_y(t), m_y(t)+ \delta], \notag \\
			&\phantom{{}=1} S_{k_b, \;l}(h) < M(l), S_{k_b, \;l}(0) < M(l)], \forall l \in [[ k_b+1, t]]) \notag \\
			&\ll  \PP(S_{k_b, \; t}(h) + v+d \in (m_y(t), m_y(t) + \delta], S_{k_b, \;l}(h) < M(l),\forall l \in [[ k_b+1, t]]) \notag \label{eqn: case3bsecondprob} \\
			&\phantom{{}=1} \PP (S_{k_b, \; t}(0) + v-d \in (m_y(t), m_y(t) + \delta], S_{k_b, \; l}(0) < M(l)], \forall l \in [[ k_b+1, t]]).
		\end{align} 
		For $v < 0$, we remove the barrier from each probability in \eqref{eqn: case3bsecondprob}. By a Gaussian estimate, we obtain 
		\begin{align}
		& \PP(S_{k_b, \; t}(h) + v+d \in (m_y(t), m_y(t) + \delta]) \cdot \PP (S_{k_b, \; t}(0) + v-d \in (m_y(t), m_y(t) + \delta]) \notag \\	&\ll   \frac{e^{- \frac{(\mu t + y-v+d)^2}{t-k_b}}}{\sqrt{t-k_b}}     \frac{e^{- \frac{(\mu t + y-v-d)^2}{t-k_b}}}{\sqrt{t-k_b}}. \notag 
		\end{align}
		For $v \geq 0$, we preserve the barriers in \eqref{eqn: case3bsecondprob} and apply the Ballot Theorem \ref{thm: ballot linear LB}:
		\begin{align} 	
			\label{eqn: case3b prob2 barrier}
			&\ll  \frac{(\mu k_b + b(k_b) - v - d+1)(\mu k_b + b(k_b) - v + d + 1) }{{(t-k_b)^2} }  \frac{e^{- \frac{(\mu t + y-v+d)^2}{t-k_b}}}{\sqrt{t-k_b}}     \frac{e^{- \frac{(\mu t + y-v-d)^2}{t-k_b}}}{\sqrt{t-k_b}}.
		\end{align}
	This concludes the estimate for the second probability in \eqref{eqn: case3beqn1}. Observe that the dominate estimates for both the first and second probability in \eqref{eqn: case3beqn1} come from $v \geq 0$. 
		
		We are now ready to estimate the contribution to $(\EE[ Z_{\delta}^2])_{k_b}$ for $ t-t^{\alpha+\epsilon} < k_b \leq t$.  Apply the two probability estimates \eqref{eqn: case3b prob1 barrier}  and \eqref{eqn: case3b prob2 barrier} to the contribution \eqref{eqn: case3beqn1}, and perform a change of variable with $\overline{v} = v- \mu k_b$:
		\begin{align}
		&(\EE[ Z_{\delta}^2])_{k_b}	\notag \\
		&\ll e^{t+t\theta} e^{t-k_b} \sum_{ -\mu k_b \leq \overline{v} < (t-k_b)\theta+y+1} \sum_d
			\frac{ ( t\theta + y+ 1) ( b(k_b) - \overline{v} + 1)}{k_b} \frac{e^{- (\overline{v} + \mu k_b)^2/ k}}{\sqrt{k_b}}\;e^{-cd^2}\notag  \\
			&\phantom{{}=1} 
			\frac{(-\overline{v}+ b(k_b) - d+1)(-\overline{v} + b(k_b) + d + 1) }{{(t-k_b)^2} }  \frac{e^{- \frac{(\mu (t-k_b) + y-\overline{v}+d)^2}{t-k_b}}}{\sqrt{t-k_b}}     \frac{e^{- \frac{(\mu (t-k_b) + y-\overline{v}-d)^2}{t-k_b}}}{\sqrt{t-k_b}}.
		\end{align}
		This simplifies to 
		\begin{align}
			&\ll  (t\theta + y + 1)  \frac{e^{t+t\theta} e^{t-k_b} e^{-2\mu^2(t-k_b) -\mu^2(k_b)  - 4\mu y}}{k_b^{3/2}(t-k_b)^3} \notag \\  \label{eqn: b} 
			&\phantom{{}=1} \sum_{-\mu k_b \leq  \overline{v} < (t-k_b)\theta+y+1} e^{2\mu \overline{v} - \frac{2(y-\overline{v})^2}{t-k_b} - \frac{\overline{v}^2}{k_b}+ 3\log ( (t-k_b)\theta + y + 1- \overline{v})} \sum_d e^{-cd^2-\frac{2d^2}{t-k_b}}.
		\end{align}
		
		The sum over $d$ is finite. To find the sum over $\overline{v}$, we apply Laplace's method. Using the fact that $\overline{v} \geq - \mu k_b$, we let $f(\overline{v}) = 2\mu \overline{v} - \frac{2(y-\overline{v})^2}{t-k_b} - \frac{\overline{v}^2}{k_b}+ 3\log ( (t-k_b)\theta + y + 1 + \mu k_b)$. 
		When $k_b=t-1$, the quantity $f(\overline{v})$ is maximal at $\overline{v_0}= y+ \frac{\mu}{4} + \frac{1}{2}$. Hence,
		
		\begin{align} 
			\int_{-\infty}^{ (t-k_b)\theta + y}e^{f(\overline{v})} \; d\overline{v} &\leq \sqrt{ \frac{2\pi}{ |f^{''}(\overline{v_0}) |}} e^{f(\overline{v_0})} 
			\ll	 \sqrt{ (t-k_b)} 	( (t-k_b)\theta +1)^3 e^{-y^2/t + 2\mu y}. \label{eqn: a}
		\end{align} 
		
		By the estimate in $(\ref{eqn: a})$, the contribution $(\ref{eqn: b})$ is
		\begin{align}
			(\EE[ Z_{\delta}^2])_{k_b} 	&\ll \label{laplace2}
		\left( 1 + \frac{y}{t^{1-\alpha}} \right)	\sqrt{(t-k_b)} 	( (t-k_b)\theta +1)^3  \notag \\
			&\phantom{{}=1} \frac{e^{t+t\theta} e^{t-k_b} e^{\log t\theta} e^{-2\mu^2(t-k_b) -\mu^2(k_b)  - 2\mu y-y^2/t }}{k_b^{3/2}(t-k_b)^3}.
		\end{align} 
		The exponential in \eqref{laplace2} simplifies to
		\begin{align}
			\frac{e^{-\theta(t-k) + (1+2 \alpha) \log (t-k_b) + (1/2 + \alpha) \log k_b + (1-\alpha)\log t} e^{- 2\mu y-y^2/t }}{k_b^{3/2}(t-k_b)^3} \ll e^{- \theta(t-k_b) }  e^{-2\sqrt{1+\theta}y-y^2/t},
		\end{align}
		since in this range, $k_b \geq t- t^{\alpha + \epsilon}$ implies $k_b \gg t$.
		
		Finally, we sum over $k_b$ and perform an index change by letting $\hat{k} = t-k_b$:
		\begin{align}
			\sum_{k= t-t^{\alpha+\epsilon}}^{t} (\EE [Z_\delta^2])_{k_b} &\ll 	\left( 1 + \frac{y}{t^{1-\alpha}} \right)e^{-2\sqrt{1+\theta}y-y^2/t} \sum_{ \hat{k}=1}^{ t^{\alpha + \epsilon} }  \sqrt{ \hat{k}}  \;( \hat{k}\theta+1)^3 e^{- \hat{k} \theta }  \notag  \\
			&\ll 	\left( 1 + \frac{y}{t^{1-\alpha}} \right)e^{-2\sqrt{1+\theta}y} e^{-y^2/t}.  \label{case3b estimate} 
		\end{align}
		This proves the case for $t-t^{\alpha+\epsilon} \leq k_b \leq t$. The estimate \eqref{case3b estimate} is $O( \EE[Z_\delta])$.
	
	
	\end{proof}
	\section{Left Tail of the Maximum} \label{Section: LeftTail of Max} 
	We prove Theorem \ref{leftbound} by the following two lemmas. 
	
	\begin{lemma}\label{lemma: lbofmax1} 
		Let $y < 0$. Let $\mu$ and $\mathcal{H}$ be as defined in \eqref{slope} and in \eqref{H}. 	For $\delta >0$, let  $$W_\delta=\#\{h \in \mathcal{H}: S_t(h) \in (\mu t + y, \mu t + y + \delta], \;S_k(h) \leq M(k); \forall k \in [1, t] \cap \mathbb{Z} \},$$ where $$M(k) = \mu k + \left(1-\frac{k}{t} \right)t^{1-\alpha}  + 1, \;\; \text{for any} \;\; 0 \leq k \leq t.$$  We have 
		$$\EE[W_{\delta}] \gg \left( 1+ \frac{1}{t^{1-\alpha}} \right)(1-y) e^{-2 \sqrt{1+\theta} y} e^{-y^2/t}.$$ In particular, $\EE[W_\delta] \to \infty$ as $t \to \infty$ and $y \to -\infty$ for fixed $\delta$.
	\end{lemma}
	\begin{lemma} \label{lemma: lbofmax2} 
		Let $W_{\delta}$ be as defined in Lemma \ref{lemma: lbofmax1}. Then 
		$$\EE[W_{\delta}^2] =  \EE[W_{\delta}]^2 + O( \EE[W_{\delta}])+ e^{-t\theta/2}O( \EE[W_{\delta}]^2).$$
	\end{lemma}
	\begin{proof}[Proofs of Lemma \ref{lemma: lbofmax1} and Lemma \ref{lemma: lbofmax2}  ] 
		The proofs are similar to the proofs for Lemma \ref{lemma: lowerboundlemma1} and Lemma \ref{lemma: lowerboundlemma2}, respectively, but with the following changes:  we use a different barrier and consider negative values of $y$.  For Lemma \ref{lemma: lbofmax2}, the dominant contribution to $\EE[W_{\delta}^2]$ is $\EE[W_{\delta}]^2 + e^{-t\theta/2}O(\EE[W_{\delta}]^2)$, which comes from the range $|h - h^{'} | > e^{t\theta/2}$. The contributions from the ranges $1< |h - h^{'} | \leq e^{t\theta/2}$ and $|h - h^{'}| \leq 1$ are $e^{-t\theta/2}O( \EE[W_{\delta}]^2)$ and $O(\EE[W_\delta])$, respectively. 
	\end{proof}

	\begin{proof}[Proof of Theorem \ref{leftbound} ] 
		By the Paley-Zygmund inequality,  we have 
		\begin{align*} 
		\PP(W_\delta \geq 1) \geq \dfrac{ (\EE[W_\delta])^2}{\EE[W_{\delta}^2]} = 1- O \left( \frac{1}{ \EE[W_{\delta}]} \right)=1 - O\left(  \frac{ e^{2\sqrt{1+\theta}y} }{ 1-y } \right), 
		\end{align*} 
		since $\EE[W_\delta] \to \infty$ as $t \to \infty$ and $y \to -\infty$.

	\end{proof}

\section{Proof of Corollary 1} \label{Section: Moments} 

\begin{proof}[Proof of Corollary \ref{momentstheorem}]
	Let  $A > 0$, and let $y >0 $ with $A > y$. 
	Consider \begin{equation} \mathcal{Z}_{\beta_c} =  \frac{1}{2e^{t\theta}}  \int_{|h|\leq 
			e^{t\theta}} \exp( \beta_c S_t(h) ) \; dh.\end{equation} 
	The aim is to show that 
	\begin{equation}\label{momentbound1}
		\PP \left( \mathcal{Z}_{\beta_c} > A (\log T)^{\frac{\beta_c^2}{4}}   \right) \ll 1/A.
	\end{equation} 
	To obtain the bound in \eqref{momentbound1}, we evaluate $\mathcal{Z}_{\beta_c}$ on a good event $E_A$ that includes a barrier. Recall that $\mu = \sqrt{1+\theta} - \frac{1+2\alpha}{4\sqrt{1+\theta}} \frac{\log t}{t}$. Define 
	\begin{equation} 
		M_A(k)= \mu k + \left(1-\frac{k}{t} \right)t^{1-\alpha} + A + 1,\;\; \text{where} \;\; 0 \leq k \leq t.
		\end{equation} 
Throughout we define the event 
	\begin{equation} \label{goodevent EA} 
		E_A=  \bigcap_{k=0}^{t-1} \left\{ \max_{|h| \leq e^{t \theta}} S_{k+1} (h) \leq M_A(k+1) \right\}.
	\end{equation}  
	The construction of $E_A$ is inspired by the first hitting time of the maximum, as shown in \eqref{upperboundmainsum}. We include the barrier $M_A(k)$ due to the branching nature of the random model. Note that the event defined in line 64 of \cite{AB} does not include a barrier.  
	
To find the probability in \eqref{momentbound1}, we consider $	\PP( \{\mathcal{Z}_{\beta_c}  > A (\log T)^{\beta_c^2/4} \} \cap E_A^c)$. It is enough to find $\PP( E_A^c).$ After replacing $y$ with $A$ in Theorem \ref{upperlowerbd}, we find that
	\begin{equation} \label{EcA} 
	 \PP( E_A^c) \ll  \left( 1 + \frac{A}{t^{1-\alpha}} \right)e^{-2 \sqrt{1+\theta}A}  e^{-A^2/t}.
	\end{equation} 
	It remains to estimate the moment $\mathcal{Z}_{\beta_c}$ on the event $E_A$. 
	
	Let $V$ be a real number. We will evaluate the moments in terms of the Lebesgue measure of high points, defined as 
	\begin{equation}\label{Lebesguemeasure1}
		S(V) = \frac{1}{2 e^{t\theta}} m\{|h| \leq e^{t\theta}: S_t(h) > V, \;S_k(h) \leq M_A(k); \; \forall k \leq t \}.
	\end{equation} 
	Let $$F(V) = \frac{1}{ 2 e^{t\theta}} m\{|h| \leq e^{t\theta}: S_t(h) < V,\; S_k(h) \leq M_A(k); \;\forall k \leq t \}.$$ We write the moments in terms of $F(V)$ and use integration by parts to obtain 	
	\begin{align} 
		\mathcal{Z}_{\beta_c} = \int_{-\infty}^{+\infty} e^{\beta_c V} dF(V) &= - e^{\beta_c V} S(V) \bigg\vert_{-\infty}^{+\infty} + \int_{-\infty}^{+\infty} \beta_c e^{\beta_c V} S(V) \;dV \notag \\
		&\ll  \int_{\beta_ct / 4}^{\mu t + A}  e^{\beta_c V} S(V) \;dV.  \label{momentsintegral}
	\end{align} 
	The boundary terms are zero: when $V= - \infty$, we use the fact that $S(V) \leq 1$; when $V= + \infty$, we have that $S(V) =0$  on $E_A$. We choose $[ \frac{ \beta_c t}{4}, \mu t + A]$ as the bounds of integration: the upper limit $\mu t +A$ follows from the restriction on the maximum on $E_A$; the lower limit $ \frac{ \beta_c t}{4}$ follows since $\int_{-\infty}^{w} e^{\beta_c V} S(V) dV \leq \int_{-\infty}^{w} e^{\beta_c V}  dV$ yields a smaller bound than desired for all $-\infty< w \leq \frac{ \beta_c t}{4}$.  The integration bounds include the optimizer $V_o=\frac{ \beta_c t}{2}$.

	To evaluate the integral \eqref{momentsintegral}, we find an upper bound for $S(V)$. Let $C > 0$ be a fixed real number. By Markov's inequality, we have that $\PP(S(V) > C\; \EE[S(V)]) \leq C^{-1}$. Let $V=\mu t + y$. To bound $S(\mu t + y)$ from above, we need a bound on its expectation. The Ballot Theorem \ref{thm: ballot linear UB} yields
	\begin{equation} \label{S(V) bound} 
		S(\mu t + y) \leq C \; \EE[S(\mu t + y)] \ll \left( 1 + \frac{A}{t^{1-\alpha}} \right)(A+ |y|) e^{-t(1+\theta)} e^{-2\sqrt{1+\theta}y} e^{-y^2/t},
	\end{equation} where $|y|=O\left(\frac{t}{\log t}\right)$. 
	
	Next, we partition the set $[ \frac{ \beta_c t}{4}, \mu t + A] \cap \sqrt{t} \mathbb{Z}$ by the sequence $(V_j, 1 \leq j \leq J)$. Let $V_0= V_1 - \sqrt{t}$ and $V_{j+1} = V_j + \sqrt{t}$. 
	Then consider 
	\begin{equation} \label{Ij}
		I_j = \int_{V_j}^{V_{j+1}} e^{\beta_c V} S(V) dV,\;\; \text{for all} \;\;  0 \leq j  \leq J,
	\end{equation}
	and let 
	\begin{equation}  \label{aj}
		a_j = A \begin{cases}
			(1+( \frac{y_j}{\sqrt{t}} )^2) &\text{if} \;\;\; y_j > 0 \\ 
			(1+ ( \frac{y_{j+1}}{\sqrt{t}} )^2)  &\text{if} \;\;\; y_{j+1} < 0 \\ 
			1	&\text{if} \;\;\; y_j < 0 < y_{j+1},\\ 
		\end{cases}
	\end{equation} 
	so that $a_j \leq (1+(\frac{y}{\sqrt t})^2)$ for $y \in [y_j, y_{j+1}]$. 
	
	Using \eqref{Ij} and \eqref{aj}, 
define the events $$E_j = \left \{ I_j \leq a_j \; \EE[I_j] \right\}, \;\; \text{for all} \;\; 0 \leq j \leq J. $$ 
	Now, let $$G = E_A \cap ( \cap_j E_j).$$ Then by Markov's inequality and by \eqref{EcA}, the probability of $G^c$ is 
	
	\begin{equation} \PP(G^c) \ll   \left(1+\frac{A}{t^{1-\alpha}} \right) e^{-2\sqrt{1+\theta}A}e^{-A^2/t} + \sum_j a_j^{-1}  \ll \frac{1}{A}.
	\end{equation}  It remains to estimate the moments on the event $G$. We thus far have 
	\begin{equation} \label{eqnmoments}
		\int_{\beta_c t/ 4 }^{\mu t+A} e^{\beta_c V} S(V) \;dV \ll \sum_j a_j \int_{V_j}^{V_{j+1}} e^{\beta_c V} \EE[S(V)] dV.
	\end{equation} 
We perform a change of variable by letting $V= \mu t + y$  (with $y_j = V_j - \mu t)$. Next, we apply the bound for $\EE[S( \mu t + y)]$ from \eqref{S(V) bound}. It follows that 
	\begin{align*} 				
		&\ll	e^{2\sqrt{1+\theta} (\mu t + y)}  e^{-t(1+\theta)}  \left(1+ \frac{A}{t^{1-\alpha}}\right) \sum_j a_j \int_{y_{j}}^{y_{j+1}}(A+ |y|) e^{-2\sqrt{1+\theta}y} e^{-y^2/t} \;dy \\
		&\ll \frac{ \left(1+ \frac{A}{t^{1-\alpha}}\right) e^{t(1+\theta)}}{t^{\alpha-1/2}} \int_{\frac{\beta_c t}{4} - \mu t }^{A} \left(1+\left(\frac{y}{\sqrt t}\right)^2\right) \frac{(A+ |y|)}{\sqrt t}  \frac{e^{-y^2/t}}{\sqrt t} \;dy \\
		&= \frac{ \left(1+ \frac{A}{t^{1-\alpha}}\right) e^{t(1+\theta)}}{t^{\alpha-1/2}} \int_{-\frac{\beta_c \sqrt{t}}{4} + o(1)}^{o(1)} (1 + |u|^2)\; |u+o(1)|e^{-u^2} du \\
		&\ll_A \frac{e^{\beta_c^2/4}}{t^{\alpha-1/2}}.
	\end{align*}
	The last line follows since $\beta_c= 2\sqrt{1+\theta}$. This proves Corollary \ref{momentstheorem}. 
	
\end{proof}

\begin{appendix}
	
	\section{Gaussian Ballot Theorem}
	Let $0<\delta\leq 1$. Define the intervals $I_x=(x,x+\delta]$ for $x>0$, $I_x=(x-\delta,x]$ if $x<0$, and $I_x=[-\delta,\delta]$ if $x=0$. The propositions are from the appendix in \cite{ADH}. We omit the proofs. 
	
	\begin{proposition}[Ballot theorem with linear barrier: Lower bound]\label{thm: ballot linear LB}
		Let $(S_j,j\geq 1)$ be a random walk with Gaussian increments of mean $0$ and variance $\sigma_j^2$ with $c^{-1}<\sigma_j<c$ for some fixed $c\geq 1$. Define the linear barrier 
		$$ 
		b(j) = aj+b(0).
		$$
		Then, for any $j\geq 1$, $(b(j)-x)\cdot b(0)\leq j$, we have
		\begin{equation}\label{ballot_inequality_lin}
			\PP \left( \forall 1\leq \ell \leq j, \ S_\ell \leq  b(\ell), \ S_j \in I_x\right)
			\gg \frac{b(0) ( b(j) -x)}{ j^{3/2} } \delta e^{- \frac{cx^2}{j} }.
		\end{equation}
	\end{proposition}

	\begin{proposition}[Ballot theorem with linear barrier: Upper bound]\label{thm: ballot linear UB}
		Let $(S_j,j\geq 1)$ be a random walk with Gaussian increments of mean $0$ and variance $1/2$. Define the linear barrier 
		$$ 
		b(j) = aj+b(0).
		$$
		Then, for any $j\geq 1$, $b(0)> 0$ and $x \leq  b(j)$, we have
		\begin{equation}\label{ballot_inequality2}
			\PP \left( \forall 1\leq \ell \leq j, \ S_\ell \leq b(\ell), \ S_j \in I_x\right)
			\ll
			\frac{(b(0)+1) (b(j) -x+1 )}{ j^{3/2} }.
		\end{equation}
	\end{proposition}

	\begin{proposition}[Ballot theorem with logarithmic barrier]\label{thm_ballot}
		Let $(S_j,j\leq t)$ be a random walk with Gaussian increments of mean $0$ and variance $1/2$. Define the logarithmic barrier 
		$$ 
		\psi_j =  \log \left(  \min (j,t-j) \right).
		$$
		Let $t/\log t\leq k\leq t$, $y=o(t/\log t)$, $-20k<x\leq \psi_k$.
		For $r=y$, we have
		\begin{equation}\label{ballot_inequality_log}
			\PP \left( \forall r<j \leq k, \ S_j\leq y + \psi_j, \ S_k \in I_x, |S_r|\leq 3r \right)
			\ll
			\frac{(y+1) (y + \psi_k -x+1 )}{ k^{3/2} } e^{- \frac{x^2}{k} }.
		\end{equation}
	\end{proposition}

\end{appendix}
\newpage

\bibliographystyle{alpha.bst}
\bibliography{refs}
\address

\end{document}